\documentclass[reqno, 12pt]{amsart}
\usepackage[utf8]{inputenc}
\usepackage{amsthm,amsfonts,amsmath,amssymb,bbm,float}
\usepackage[numbers]{natbib}
\usepackage{xcolor}
\usepackage{hyperref}
\hypersetup{
    colorlinks,
    linkcolor={blue!80!black},
    citecolor={red!90!black},
    urlcolor={green!80!black}
}
\usepackage[margin=1in]{geometry}
\usepackage{mathtools}
\usepackage{etoolbox}
\usepackage{thmtools}
\usepackage{enumitem}

\newtheorem{theorem}{Theorem}[section]

\newtheorem{proposition}[theorem]{Proposition}
\newtheorem{corollary}[theorem]{Corollary}
\newtheorem{lemma}[theorem]{Lemma}
\newtheorem{remark}[theorem]{Remark}

\newcommand{\bd}[1]{\boldsymbol{\mathbf{#1}}}

\providecommand\given{}

\DeclarePairedDelimiter\fbr{(}{)}

\DeclarePairedDelimiterX\abs[1]\lvert\rvert{\ifblank{#1}{\:\cdot\:}{#1}}
\DeclarePairedDelimiterX\Set[1]\{\}{\renewcommand\given{\SetSymbol[\delimsize]}#1}

\DeclarePairedDelimiterXPP\Indicator[1]{\mathbbm{1}}[]{}{#1}

\DeclarePairedDelimiterXPP\Prob[1]{\mathbb{P}}(){}{\renewcommand\given{\nonscript\:\delimsize\vert\nonscript\:\mathopen{}}#1}

\DeclarePairedDelimiterXPP\Exp[1]{\mathbb{E}}(){}{\renewcommand\given{\nonscript\:\delimsize\vert\nonscript\:\mathopen{}}#1}

\DeclarePairedDelimiterXPP\expo[1]{\exp}(){}{\renewcommand\given{\nonscript\:\delimsize\vert\nonscript\:\mathopen{}}#1}

\DeclarePairedDelimiterXPP\Var[1]{\mathrm{Var}}(){}{\renewcommand\given{\nonscript\:\delimsize\vert\nonscript\:\mathopen{}}#1}

\DeclarePairedDelimiterXPP\Varhat[1]{\widehat{\mathrm{Var}}}(){}{\renewcommand\given{\nonscript\:\delimsize\vert\nonscript\:\mathopen{}}#1}

\DeclarePairedDelimiterXPP\Cov[1]{\mathrm{Cov}}(){}{\renewcommand\given{\nonscript\:\delimsize\vert\nonscript\:\mathopen{}}#1}

\newcommand{\deri}{\mathrm{d}}
\newcommand{\RR}{\mathbb{R}}
\newcommand{\bbR}{\mathbb{R}}

\newcommand{\bbE}{\mathbb{E}}

\newcommand{\Letahat}{\widehat{\bd{L}^{(\eta)}}}
\newcommand{\Leta}{\bd{L}^{(\eta)}}
\newcommand{\Sigmainvhat}{\widehat{\bd{\Sigma}^{-1}}}
\newcommand{\Sigmainv}{\bd{\Sigma}^{-1}}

\newcommand{\leta}{l^{(\eta)}}

\newcommand{\bmt}{\mathrm{BMT}}
\let\hat\widehat

\title{Learning Networks from Gaussian Graphical Models and Gaussian Free Fields}

\author[S. Ghosh]{Subhro Ghosh$^*$}
\address{Subhro Ghosh, \ Department of Mathematics, \ National University of Singapore.}
\email{subhrowork@gmail.com}

\author[S. S. Mukherjee]{Soumendu Sundar Mukherjee$^*$}
\address{Soumendu S. Mukherjee, \ Theoretical Statistics and Mathematics Unit, \ Indian Statistical Institute, Kolkata}
\email{ssmukherjee@isical.ac.in}

\author[H.-S. Tran]{Hoang-Son Tran$^*$}
\address{Hoang-Son Tran, \ Department of Mathematics, \ National University of Singapore}
\email{hoangson.tran@u.nus.edu}

\author[U. Gangopadhyay]{Ujan Gangopadhyay}
\address{Ujan Gangopadhyay, \ Department of Mathematics, \ National University of Singapore}
\email{ujan@nus.edu.sg}

\usepackage{tcolorbox}
\begin{document}
\begin{abstract}
We investigate the problem of estimating the  structure of a weighted network from repeated measurements of a Gaussian Graphical Model (GGM) on the network. In this vein, we consider GGMs whose covariance structures align with the geometry of the weighted network on which they are based. Such GGMs have been of longstanding interest in statistical physics, and are referred to as the Gaussian Free Field (GFF). In recent years, they have attracted considerable interest in the machine learning and theoretical computer science. In this work, we propose a novel estimator for the weighted network (equivalently, its Laplacian) from repeated measurements of a GFF on the network, based on the  Fourier analytic properties of the Gaussian distribution. In this pursuit, our approach exploits complex-valued statistics constructed from observed data, that are of interest on their own right. We demonstrate the effectiveness of our estimator with  concrete recovery guarantees and bounds on the required sample complexity. In particular, we show that the proposed statistic achieves the parametric rate of estimation for fixed network size. In the setting of networks growing with sample size, our results show that for Erdos-Renyi random graphs $G(d,p)$ above the connectivity threshold, we demonstrate that network recovery takes place with high probability as soon as the sample size $n$ satisfies $n \gg d^4  \log d \cdot p^{-2}$.
\end{abstract}

\def\thefootnote{*}\footnotetext{Equal\,contribution}



\maketitle

\tableofcontents

\section{Introduction}

\color{red}

\color{black}

\subsection{Gaussian Graphical Models}
Gaussian Graphical Models (GGM),
also known as 
Gaussian Markov Random Fields (GMRF),  
 are 
multivariate Gaussian distributions 
defined on undirected graphs. 
In these models, 
a Gaussian random variable is associated 
to each vertex of a graph,
and the existence or non-existence of 
edges captures the dependency structure
between these random variables. 

More precisely, 
suppose we have a graph 
$G=(V,E)$ 
on $|V|=d$ vertices. 
In a GMRF/GGM model,
we assume our sample 
$\bd{X}_1,\ldots,\bd{X}_n$ 
consists of i.i.d.\ copies of a random vector 
$\bd{X}=(X_1,\ldots,X_d)^\top$ 
having a multivariate normal distribution.

It is well-known that the covariance 
matrix $\bd{\Sigma}$ of a GGM captures the 
dependency structure between $X_i$s.
That is, for $i\neq j$, the $(i,j)$-th matrix entry $\bd{\Sigma}[i,j]=0$ 
if and only if the variable $X_i$ is independent of the
variable $X_j$. 

The inverse of the covariance matrix, 
often called the \textit{precision matrix} or 
the information matrix, is also of
longstanding interest, because it captures the 
conditional dependency structure between
$X_i$s. More precisely, the $(i,j)$-th matrix entry 
$\bd{\Sigma}^{-1}[i,j]$ for $i\neq j$ is $0$
if and only if $X_i$ is independent of $X_j$ 
given the rest of the variables.

More generally,the following may be shown to hold true. If, for some subset of
indices $A \subset V$, it holds that removing the vertices in 
$A$ from the graph $G$ disconnects the vertices $i$ and $j$, then the random variables
$X_i$ and $X_j$ are conditionally independent given $(X_k)_{k \in A}$. This is a form of graphical \textit{Markov property},
which  endows GGM-s with a highly attractive structure, both as a stochastic model and as a data modelling tool.

Unsurprisingly, GGM-s have attracted interest in a wide array of application domains as an effective modelling technique to capture the dependency structures among variates. These include applications to genomics (\citet{Genetics1},
\citet{Genetics2}, \citet{Genetics3}, \citet{Genetics4}); neuroscience  (\citet{Huang2010}, \citet{GAPT10}, \citet{NIPS09}, \citet{StrokeModelling10}); causal inference (\citet{Causal1}); to name a few.

\subsection{Learning networks from random fields}
The problem of learning a network from observations of a random field that lives on it has been a topic of great interest in recent years. A significant instance of this is accorded by the Ising model on graphs. The Ising model on a graph $G$ is a random field with values in the set $\{+1 , -1\}$, with a dependency structure that reflects the structure of the graph. Estimating the underlying graph (or various properties thereof) from observations of the Ising model has been a topic of intensive research activities in recent years. Investigations on this problem have been carried out in various settings, for details we refer the interested reader to (\citet{Ising1,Ising2,Ising3,Ising4,Ising5}) for a partial list of references.

In this paper, we investigate the problem of learning a network from a Gaussian Graphical Model supported on the network. More generally than unweighted graphs, we will concern ourselves with the broader problem of estimating a weighted network from a GGM supported thereon. This necessitates the correlation structure of the GGM to carry information about the weights on the edges of the network. A canonical choice for such a GGM is proffered by the so-called \textit{Gaussian Free Field} (\textit{abbrv.} GFF), which is what we will focus on in this work. 

\subsection{Gaussian Free Fields}
 Gaussian Free Fields (abbrv. GFF) have emerged as important models of strongly correlated Gaussian fields, that are canonically equipped to capture the geometry of their ambient space. In the case of graphs, the background geometry is encapsulated in the graph Laplacian. GFF-s  arise originally in theoretical physics, in the study Euclidean quantum field theories. Applications in physics generally require the GFF to be defined on the continuum, which is often a challenge in itself in context of the fact that, even on Euclidean spaces of dimension $>1$, the continuum GFF is defined only as a distribution valued random variable. On graphs and weighted networks, the setting we are principally interested in, it is of interest to study the so-called \textit{Discrete Gaussian Free Field} (\textit{abbrv.} DGFF). The DGFF, in comparison to the continuum setting, is well-defined as a random field that lives on the nodes of the network; however, it exhibits degeneracy properties as a Gaussian random vector, which demands some consideration in its definition.

The GFF is, for many natural reasons, a GGM of wide interest in its own right. For one, a quadratic form based on the network Laplacian encodes smoothness with respect to the geometry of the weighted network. This underpins the significance of GFF based models in active and semi-supervised learning (\citet{GFF-ML-1,GFF-ML-2,GFF-ML-3,ghosh2022learning}). In fact, it has been shown (\citet{Kelner}) that DGFFs essentially cover all possibilities in an important class of GGMs known as \textit{attractive} GGMs, wherein the pairwise correlations are all non-negative and which arise naturally in a wide array of applications such has phylogenetic studies and copula models of finance. For more details on the generalities of the GFF and its significance with statistical and mathematical physics,  we refer the reader to the excellent mathematical surveys (\citet{Sheffield,Berestycki}).

 \subsubsection{The Massless DGFF}
The massless DGFF on a weighted graph $G=(V(G),E(G))$
is defined as follows. Let $\partial$ be a distinguished set of vertices, 
called the boundary of the graph.
Let $S_n$ be the simple symmetric random walk on $G$.
Let $\tau$ be the hitting time of $\partial$.
The Green function $\mathbf{G}(x,y)$ is defined for $x,y\in V(G)$
by putting 
\[
\mathbf{G}(x,y) = \frac{1}{\deg(y)} \bbE\Bigl(\sum_{n=0}^\infty\Indicator{X_n = y; \tau>n}\Bigr)\;.
\]
The DGFF is the 
centered Gaussian vector $(h(x))_{x\in V(G)}$
with covariance given by the Green function $\mathbf{G}$.
In other words, if $A \subset \mathbb{R}^{|V(G)|}$, then
the probability distribution  of 
$(h(x))_{x\in V(G)}$ is given by 
\begin{equation}
\Prob[\Big]{(h(x))_{x\in V(G)} \in A}
=
\frac{1}{Z}
\int_A\expo[\Big]{-\frac{1}{4}\sum_{x_i\sim x_j} (h(x_i)-h(x_j))^2}
\prod_{x_i\not\in\partial}
\mathrm{d}h(x_i)
\prod_{x_i\in\partial}
\delta_0(\mathrm{d}h(x_i))\;,
\end{equation}
where $\delta_0$ is the Dirac delta measure at 0. In particular, the field always takes value $0$
at the boundary vertices $\partial$.

 \subsubsection{The Massive DGFF}
The massive DGFF on a weighted graph $G=(V(G),E(G))$ is, in a sense, simpler than the massless case, and is defined as follows. For the \textit{mass parameter} $\mu$, the massive DGFF $(h(x))_{x\in V(G)}$ on $G$ is given by the following relation. For $A \subset \mathbb{R}^{|V(G)|}$, then
the probability distribution  of 
$(h(x))_{x\in V(G)}$ is given by 
\begin{equation}
\Prob[\Big]{(h(x))_{x\in V(G)} \in A}
=
\frac{1}{Z}
\int_A\expo[\Big]{-\frac{1}{4}\sum_{x_i\sim x_j} (h(x_i)-h(x_j))^2 + \frac{1}{4} \mu \sum_{x_i \in V} h(x_i)^2}
\prod_{x_i \in V}
\mathrm{d} h(x_i)
\;.
\end{equation}
In other words, the probability density function of $\mathbf{h}:=(h(x))_{x\in V(G)}$ is given, up to a normalization constant, by 
\begin{equation}
\expo[\Big]{- \frac{1}{4} \bd{h}^\top (\bd{L} + \mu \bd{I}) \bd{h} }
\;,
\end{equation}
where $\bd{L}$ is the standard Laplacian matrix of the graph, and $\mu$ is the mass parameter of the model.

The massive DGFF is the setting we will concern ourselves with in this article. The mass parameter $\mu$ clearly controls the well-conditionedness of the model; and the regime we will particularly concern ourself with is the one where the parameter $\mu$ is not too large, thereby focusing on poorly conditioned GGMs.

\subsection{Learning networks from Gaussian Free Fields}
In this article we consider the 
problem of precision matrix estimation of the Gaussian Free Field. To wit, we consider a graph $G=(V(G),E(G))$ with vertex set $V(G)$ and edge set $E(G)$. Let $|V(G)| = d$ be the number of 
vertices and we label the vertex set with $\{1,\ldots,d\}$.
Suppose we have an i.i.d.\ sample $\bd{X}_1,\ldots,\bd{X}_n$
of random vectors in $\RR^d$ where each $\bd{X}_m$ follows a  
multivariate normal distribution $\mathcal{N}
(\bd{0},\bd{\Sigma})$. 
The covariance matrix
is related to the graph $G$ in the following way. 
We assume $\bd{\Sigma} = (\bd{L}+\mu\bd{I})^{-1}$
where $\bd{L}$ is the graph Laplacian and $\mu>0$ is an
unknown parameter. 

The graph Laplacian $\bd{L}$ is the 
$d\times d$ matrix defined as $\bd{L}=\bd{D}-\bd{A}$ 
where $\bd{D}$ is the degree matrix and $\bd{A}$ is the 
adjacency matrix. Thus, for an unweighted graph, $(i,i)$-th entry of $\bd{L}$
is the degree of vertex $i$, and for $i\neq j$ the 
$(i,j)$-th entry of $\bd{L}$ is $0$ if $i$ is not connected to $j$, and it is $-1$ if 
$i$ is connected to $j$. For a weighted graph, $\bd{A}$ is the \textit{weighted} adjacency matrix of the
graph; i.e. $ \bd{A}_{ij} = w_{ij}$, where $w_{ij}$ is the weight of the edge between vertices $i$ and $j$. The matrix $\bd{D}$ in this setting would be the weighted degree matrix of the graph; i.e., $\bd{D}$ is a diagonal matrix with $\bd{D}_{ii}=\sum_{j \in V(G)} \bd{A}_{ij}$.

We will concern ourselves the problem of 
estimating $\bd{L}$  from $\bd{X}_1,\ldots,\bd{X}_n$.

Thus our problem is same as estimating
the graph underlying a DGFF where $\mu$ 
can be thought of as the number of vertices
that are designated as the boundary. 
These boundary vertices are connected 
with all the other vertices.

\subsection{A survey on the estimation of precision matrices}
In view of its multifaceted importance, \cite{Dempster72} initiated
investigations into the problem of estimating the precision matrix of a GGM.
In this subsection, we provide a partial survey of the principal approaches to this 
longstanding problem, and the associated problem of estimating the covariance matrix of a GGM.

\subsubsection{Estimation of covariance matrices}
Estimating the covariance matrix of a GGM is 
not difficult when the sample size $n$
is much bigger than $d$. In this case 
the sample covariance matrix 
\[
\bd{\Sigma}_n \coloneqq 
\frac{1}{n}
\sum_{k=1}^n
\Bigl(\bd{X}_k-\overline{\bd{X}}_n\Bigr)
\Bigl(\bd{X}_k-\overline{\bd{X}}_n\Bigr)^\top\;,
\mbox{ where }
\overline{\bd{X}}_n
\coloneqq
\frac{1}{n}\sum_{k=1}^n\bd{X}_k\;,
\]
is a natural estimator of 
$\bd{\Sigma}$, 
and has good consistency properties.
In the high-dimensional setup i.e., 
when the number of variables $d$ is 
much larger than $n$, the sample
covariance matrix is not a 
good estimator of the true covariance 
matrix $\bd{\Sigma}$.
In this case estimating the 
covariance matrix
is a challenging problem,
estimating the precision 
matrix is even more difficult.

\subsubsection{Estimation under order structures}
It is possible to estimate $\bd{\Sigma}$ 
using $\bd{\Sigma}_n$ consistently 
if there exists additional structure
on the variables. For example,
if the variables have a certain 
total ordering, for example in a 
time series data, then one 
may assume $\bd{\Sigma}[i,j]$ 
is $0$ or near $0$ when $|i-j|$ 
is big enough.
This leads to a 
banding structure in 
$\bd{\Sigma}$. Under this kind of 
assumptions \citet{BL08a} 
showed that banding the 
sample covariance matrix 
leads to a consistent 
estimator. \citet{CZZ10} 
considered the same class 
of estimators as 
\citet{BL08a} and 
established the minimax 
rate of convergence and 
also constructed a 
rate-optimal estimator. The minimax 
rate is given by
\[
\inf_{\widehat{\bd{\Sigma}}}
\sup_{\mathcal{P}_\alpha}
\Exp{\|\widehat{\bd{\Sigma}} -
\bd{\Sigma}\|^2}
\asymp 
\min\Bigl\{ n^{-2\alpha/(2\alpha+1)}+
\frac{\log d}{n},\frac{d}{n} 
\Bigr\}\; \cdot
\]
Here $\alpha$ is a sparsity parameter
and $\mathcal{P}_\alpha$ is a class 
of sparse covariance matrices.
Larger $\alpha$ corresponds to 
sparser matrices. 

Assuming certain ordering 
structures on the
variables, methods based on 
banding the Cholesky
factor of the inverse 
covariance matrix for 
estimating the covariance 
matrix have also been 
proposed and studied (see, 
e.g., \citet{WP03}, 
\citet{HLPL06}).

\subsubsection{Estimation under structured sparsity}
A natural total ordering on 
the set of variables is 
unavailable in many 
situations.
Also finding a suitable 
basis under which the 
sample-covariance matrix 
displays a banding 
structure is often 
computationally impractical.
To deal with these situations
\citet{K08} and \citet{BL08b} 
suggested assuming certain 
permutation invariant sparsity 
conditions on the covariance matrix
and proposed thresholding the 
sample covariance matrix for estimation. 
They obtained rates of convergence 
for the thresholded sample 
covariance estimator.

\subsubsection{Penalized likelihood based methods}
Penalized likelihood based methods
are also very popular for estimation
of sparse precision matrices. 
\citet{MB06} estimate a sparse 
precision matrix by fitting a lasso 
model to each variable, using the 
others as predictors. 
$\bd{\Sigma}^{-1}[i,j]$ is then
estimated to be nonzero if either the
estimated coefficient of variable $i$ on
$j$ or the estimated coefficient of
variable $j$ on $i$ is nonzero. They show
that asymptotically, this consistently
estimates the set of nonzero elements of
$\bd{\Sigma}^{-1}$.

\subsubsection{The Graphical Lasso}
Several algorithms for 
the exact
maximization of the 
$\ell_1$-penalized
log-likelihood have been
proposed in the literature
(see for e.g. \citet{YL07}, \citet{DBG07}, \citet{DVR08}, \citet{FHT07}, \citet{RBLZ08}, \citet{CLL11}, \citet{CLLX12}, \citet{RWRY11}.) 
\citet{FHT07} introduced
the Graphical Lasso estimator 
which minimizes
\[
-\log\det\bigl(\bd{\Sigma}^{-1}\bigr) 
+ \frac{1}{n}
\sum_{i=1}^n
\bigl(\bd{X}_i-\bd{\mu}\bigr)^\top
\bd{\Sigma}^{-1}
\bigl(\bd{X}_i-\bd{\mu}\bigr)
\]
subject to the sparsity condition $\sum_{i\neq j}\bd{\Sigma}^{-1}[i,j]
\leq t\;,$ where $t\geq 0$ is a tuning parameter.
An equivalent formulation is minimizing
\[
-\log\det\bigl(\bd{\Sigma}^{-1}\bigr)
+\mathrm{trace}\bigl(\bd{\Sigma}^{-1}\bd{\Sigma}_n\bigr)
+\lambda\sum_{i\neq j}
\bigl|\bd{\Sigma}^{-1}[i,j]\bigr|
\]
where $\lambda\geq 0$ is a tuning parameter.
If $\bd{\Sigma}^{-1}$ satisfies various conditions, which typically include an assumption similar to or stronger than the restricted eigenvalue (RE) condition (a condition which, in particular, lower bounds
the smallest eigenvalue of any 
$2d\times 2d$ principal submatrix of $\bd{\Sigma}$ where $d$ is the maximum vertex degree) then Graphical Lasso succeeds in recovering the graph structure. Further, under some incoherence assumptions on the precision matrix (stronger than RE), it has been shown by  
\citet{RWRY11} that the sparsity pattern of the precision matrix can be accurately recovered from 
$O((1/\alpha^2)d^2\log(n))$ samples; here $\alpha\in(0,1)$ is the incoherence parameter defined as follows. Suppose $\bd{\Theta}=\bd{\Sigma}^{-1}$. Let
\[
\bd{\Gamma}
\coloneqq
\nabla^2_{\bd{\Theta}^\prime}
g(\bd{\Theta}^\prime)|_{\bd{\Theta}^\prime
=\bd{\Theta}} 
= \bd{\Theta}^{-1} \otimes \bd{\Theta}^{-1}
\]
where 
\[
g(\bd{A}) \coloneqq 
\begin{cases}
-\ln\det \bd{A} & \mbox{ if } A\succ 0\\
\infty & \mbox{ otherwise }
\end{cases}
\]
and $\otimes$ denotes the Kronecker 
matrix product. Thus $\bd{\Gamma}$ is a 
$p^2\times p^2$ matrix,
indexed by vertex pairs
so that $\bd{\Gamma}[(j,k),(l,m)]$ is the partial derivative 
\[
\frac{\partial^2 g}{\partial \bd{\Theta}^\prime_{jk} \partial \bd{\Theta}^\prime_{lm} }
\]
evaluated at $\Theta$. 
For Gaussian observations this is simply $\Cov{X_jX_k,X_lX_m}$.
The mutual incoherence or irrepresentable condition is the following:
\[
\|\bd{\Gamma}_{S^c S} (\bd{\Gamma}_{SS})^{-1}\|_{\infty}
\leq 1-\alpha
\]
for some $\alpha\in(0,1)$. 
Here $S$ is the set of vertex pairs either of the form $(i,i)$ or $(i,j)$
if $i$ is connected to $j$. 
This condition imposes control on the
influence that the non-edge terms, 
indexed by $S^c$, can have on the 
edge-based terms, indexed by $S$. 
A similar condition for the Lasso,
with the covariance matrix $\bd{\Sigma}$ in the place of $\bd{\Gamma}$, is necessary and sufficient for support recovery 
using the ordinary Lasso 
(\cite{MB06}, \cite{T06}, \cite{W09}, \cite{ZY06}.)

\subsubsection{The CLIME estimator}
Another popular estimator is the CLIME 
(constrained $\ell_1$-minimization for
inverse matrix estimation) introduced by
\citet{CLL11}. It solves the following
optimization problem
\[
\mbox{minimize } 
\|\bd{\Theta}\|_1
\mbox{ such that }
\|\widehat{\bd{\Sigma}_n}\Theta - \bd{I}\|_\infty \leq \lambda\;, 
\]
where $\lambda$ is a tuning parameter. The analysis of CLIME uses a condition
number assumption. For example,
if entries $\bd{\Sigma}^{-1}[i,j]$ are either zero or bounded away from zero by an absolute constant then CLIME succeeds at structure recovery when given roughly $C M^4 \log d$ samples where $M=\max_{\|\bd{u}\|_\infty\leq 1}\|\bd{\Sigma}^{-1}\bd{u}\|_\infty$. 
\citet{CLZ16} obtained minimax rates for
precision matrix estimation in the high-dimensional setting. They also proposed 
a fully data driven estimator called
ACLIME based on adaptive constrained 
$\ell_1$ minimization and obtained
rate of convergence. A survey of  
minimax rates for sparse covariance
matrix estimation and sparse precision 
matrix estimation along with rates of 
convergence of various $\ell_1$-penalized
estimators can be found in the expository
article \citet{CZZ16}. \citet{DBG08}
considers penalizing the number of
nonzero terms instead of the $\ell_1$
penalty. \citet{LLW09} have showed that
for a class of non-Gaussian distribution
called nonparanormal distribution,
the problem of estimating the graph 
also can be reduced to estimating the
precision matrix. \citet{Y10} replaced
the lasso selection by a Dantzig-type
modification, where first the ratios
between the off-diagonal elements
$\bd{\Sigma}^{-1}[i,j]$ and the 
corresponding diagonal element
$\bd{\Sigma}^{-1}[i,i]$ were estimated for
each row $i$ and then the diagonal
entries $\bd{\Sigma}[i,i]$ were obtained
given the estimated ratios. \citet{LF07}, 
\citet{FFW09} considered penalizing the
normal likelihood with a nonconvex 
penalty in order to reduce the bias of 
the  $\ell_1$ penalized estimator.

\subsubsection{Information theoretic lower bounds}
\citet{MVL20} considers the problem
of finding information theoretic lower bound on the sample size for recovering the precision matrix in a sparse GGM. They establish that
for a model defined on a sparse graph with $p$ nodes, a maximum degree $d$ and minimum normalized edge strength $\kappa$, the necessary number of samples scales at least as $d \log p/\kappa^2$. The 
parameter $\kappa$, 
called the minimum normalized edge strength,
is defined as
\[
\kappa
\coloneqq
\min_{(i,j)\in E} 
\frac{\bd{\Sigma}^{-1}[i,j]}{\sqrt{\bd{\Sigma}^{-1}[i,i]}
\sqrt{\bd{\Sigma}^{-1}[j,j]}}\;.
\]
They propose an algorithm called
Degree-constrained
Inverse Covariance Estimator (DICE) which achieves this information theoretic lower bound.
They also propose another algorithm called 
Sparse Least-squares Inverse Covariance Estimator
(SLICE) which uses mixed integer quadratic programming, making it more efficient, but the sample complexity of SLICE is roughly $1/\kappa^2$ higher than the information theoretic lower bound.

\subsubsection{Ill-conditioned GGMs}
CLIME or Graphical Lasso is
only suitable when the precision 
matrix is well-conditioned. 
\citet{KKMM20} considers the problem of estimating a ill-conditioned precision matrix in some important class of
GGMS. They give fixed polynomial-time algorithms for learning \emph{attractive GGMs} and \emph{walk-summable GGMs} 
with a logarithmic number of samples. Attractive GGMs are GGMs in which the 
off-diagonal entries of $\Theta$ are 
non-positive. This means that all partial
correlations are non-negative. These 
are often used in practice, for example 
in phylogenetic applications, observed
variables are often positively dependent because of shared ancestry (see \citet{Z16}); also in finance where using a latent global market variable leads to positive dependence (see \citet{MS05}). \citet{KKMM20} introduces an algorithm called GREEDY-AND-PRUNE which has sample complexity $\log(1/\kappa)$ times higher than the information theoretic lower bound but is more efficient than other methods. Walk-summable GGMs are defined as follows, see \citet{MJW06} for more details.
A walk of length $l\geq 0$ in
a graph $G=(V(G),E(G))$ is a 
sequence $w=(w_0,w_1,\ldots,w_l)$
of nodes $w_k\in V(G)$ such that 
each step of the walk, 
say $(w_k,w_{k+1})$,
corresponds to an edge of the graph
$\{w_k,w_{k+1}\}\in E(G)$. 
Walks may visit nodes and
cross edges multiple times. 
We let $l(w)$ denote the 
length of walk $w$.
We define the weight of a 
walk to be the product of 
the edge weights along 
the walk:
\[
\phi(w)=\prod_{k=1}^{l(w)} 
r_{w_{k-1},w_k}\;.
\]
We also allow zero-length ``self''
walks $w=(v)$ at each node $v$ for 
which we define $\phi(w)=1$.
The connection between these walks
and Gaussian inference can be seen
as follows. We decompose the
covariance matrix as
\[
\bd{\Sigma} = 
\bd{\Theta}^{-1} = 
(\bd{I}-\bd{R})^{-1} 
= \sum_{k=0}^\infty \bd{R}^k,
\]
for $\rho(\bd{R})<1$.
We have assumed that the model is normalized by rescaling variables so that $\Theta[i,i]=1$ for all $i$. Then then $\bd{R}=\bd{I}-\bd{\Theta}$ has
zero diagonal and the off-diagonal elements are equal to the partial 
correlation coefficients $r_{ij}$
\[
r_{ij}\coloneqq
\frac{\Cov{x_i,x_j\given x_{V\setminus\{ij\}}}}
{\sqrt{\Var{x_i\given x_{V\setminus\{ij\}}}}
\sqrt{\Var{x_j\given x_{V\setminus\{ij\}}}}}
\]
The $(i,j)$'th entry of $\bd{R}^l$ 
is sum over weights of paths of 
length $l$ that go from $i$ to $j$.
A GGM is called walk-summable if for 
all $i,j$ the sum of $|\phi(w)|$ over all walks from $i$ to $j$ is finite almost surely. For learning walk-summable GGMs \cite{KKMM20} introduces an algorithm called HYBRIDMB which has sample complexity $1/\kappa^2$ times the information theoretic lower bound.

\subsubsection{Bayesian Approaches}
Bayesian methods have also been utilized
for estimation of precision matrices.
\citet{BG15} considers a prior distribution on the off-diagonal entries of the precision matrix which put a mixture of a point mass at zero and certain absolutely continuous distribution. They establish posterior consistency of the resulting estimator.
Posterior consistency was established
for class of banded precision matrix in
\citet{BG14}. In \citet{SGM21} the authors
consider the situation where the observation from the Graphical model are tampered with Gaussian measurement errors.

\subsection{Our contributions and future directions} In this work, we contribute a novel estimator for the weighted network from  samples of a GFF on the network,
based on certain Fourier analytic properties of the Gaussian distribution. In this pursuit,
our approach exploits complex-valued statistics constructed from observed data, that are of
interest on their own right. 

To wit, we begin with the observation that the logarithm of a probability density in a Gaussian Free Field is essentially a quadratic form of the network Laplacian (up to an additive constant), and is thus of interest in learning the Laplacian via its quadratic forms. Fundamentally, our approach is underpinned by the observation that the standard Gaussian density is essentially a fixed point of the Fourier transform, and for a general covariance matrix $\Sigma$, the Fourier transform entails a mapping $\Sigma \mapsto \Sigma^{-1}$ on the exponent of the Gaussian density. Thus, two successive applications of the Fourier transformation acts like an involution on a Gaussian density, up to a normalization constant. While taking the Fourier transform via numerical integration can be challenging computationally, we tackle this via a stochastic approach, by averaging against an independent Gaussian with a suitable dispersion. Our test statistic, which is complex-valued, is conveniently bounded in absolute value by 1, thereby allowing for superior concentration of measure effects. 

Complex-valued statistical observables have certain salutary properties, which can be of interest from theoretical perspectives. In particular, they embody a phase, which can lead to stronger cancellation effects due to the destructive interference of phases. Such statistics, however, have not been exploited to their full potential, and literature on their application is quite limited. A recent instance in the literature is accorded by \citet{Belomestny2019SparseDeconvolution}, where the authors use complex-valued test statistics in order to perform deconvolution in the context of covariance matrix estimation. 

The approach proposed in the present work is conceptually simple and computationally light, in addition to having highly tractable analytical properties. Most of the known techniques for precision matrix estimation for GGMs are known to be of limited effectiveness when the GGM is ill-conditioned, such as the GFF (with a small mass parameter). Furthermore, the known techniques for graph recovery from GGMs, especially in the ill-conditioned setting (see, e.g., \citet{Kelner}) are often combinatorial in nature and are more suited for the setting of unweighted graphs. Much of the existing literature is also geared towards the learning of sparse graphs (in the context of the high-dimensional tradeoff between system size and data availability), and involve computationally intensive optimization procedures. However, it may be noted that a network can be \textit{low dimensional} without being sparse -- this is  significant vis-a-vis current interest in generative models, where a dense network can be generated from generative model with only a few parameters. A classic case in point is that of the Erdos-Renyi random graph in the dense regime, which is characterised by a single parameter, namely the edge connection probability; another instance on similar lines is provided by a stochastic block model.

Our approach addresses many of these issues with a simple and easy-to-compute estimator. We demonstrate the effectiveness of our estimator with concrete
recovery guarantees and bounds on the required sample complexity. In particular, we show
that the proposed statistic achieves the parametric rate of estimation for fixed network size.
In the setting of networks growing with sample size, our results show that for Erdos-Renyi
random graphs $G(d,p)$ above the connectivity threshold, network recovery takes place with high probability as soon as the sample size $n$ satisfies $n \gg d^4 \log d \cdot p^{-2} $.

We believe that the present work inaugurates the study of complex-valued statistics and techniques inspired by Gaussian Fourier analysis in the context of GGMs, and more generally, Gaussian random fields with a geometric structure. A direction of particular interest would be to augment our simple approach with additional ingredients so as to account for structured network models, a case in point being that of sparsity. Further improvisations and modifications of our relatively straightforward approach to provide efficient learning in wider classes of networks and enhanced rates in specific structural scenarios provide natural avenues for further investigation. 

{\bf Acknowledgements.} S.G. was supported in part by the MOE grants R-146-000-250-133, R-146-000-312-114 and MOE-T2EP20121-0013. S.S.M. was partially supported by an INSPIRE research grant (DST/INSPIRE/04/2018/002193) from the Department of Science and Technology, Government of India.

\section{Learning networks from Gaussian Free Fields}

In this section, we will lay out the our approach to estimation of the weighted network (equivalently, its Laplacian) based on Gaussian Fourier analysis. To fix notations, we set
$\bd{\Sigma} = (\bd{L} + \mu \bd{I})^{-1}$ and $\bd{L}$ are $d\times d$ matrices and $\mu>0$. Let $$\lambda_1 := \lambda_{\max}(\bd{L}).$$

For $\eta>0$, we define
\begin{equation} \label{eq:Leta-def}
\Leta
\coloneqq 
\bigl(\bd{\Sigma}+\eta^{-1}\bd{I}\bigr)^{-1}\;.
\end{equation}

Our estimation procedure will comprise of two steps. We will first estimate $\Leta$, and then from this estimate of $\Leta$ we will construct an estimator of $\bd{\Sigma}^{-1}$.

\subsection{Estimation of \texorpdfstring{$\Leta$}{Leta}}
Our approach to estimation of $\Leta$ is based on the Fourier analytic properties of the Gaussian distribution. In particular, up to constants, the standard Gaussian density in $d$ dimensions is a fixed point of the Fourier transform. For a general covariance matrix $\bd{\Sigma}$, the Fourier transform induces a mapping $\bd{\Sigma} \mapsto \bd{\Sigma}^{-1}$ on the exponent of the Gaussian density. Thus, two successive applications of the Fourier transformation acts like an involution on a Gaussian density, up to a normalization constant. 

Therefore, if there is a simple way to mimic the Fourier transform of an underlying Gaussian based on data generated from that distribution, then a two-fold application of the Fourier transform (followed by a logarithmic transformation) would approximately result in a quadratic form in the precision matrix. Obtaining the precision matrix from this quadratic form would then be a rather simple matter. 

In expectation, taking the Fourier transform of a Gaussian density on the basis of random samples from it is relatively canonical : we simply consider the plane wave corresponding to the random variable (or more precisely, its empirical version). To wit, if $\bd{W}$ is a $d$-dimensional Gaussian random vector, then $\Exp[\Big]{\expo[\big]{i\langle \bd{\xi} , \bd{W} \rangle}}$ is the Fourier transform of the density of $\bd{W}$, evaluated at $\bd{\xi} \in \mathbb{R}^d$. In practice, we do not work at the level of expectations, but based on a large set of samples drawn from $\bd{W}$; thus the act of taking expectation is substituted canonically with averaging over this sample. Thus, we are performing a \textit{stochastic} version of a Fourier transformation, based on observed samples from a distribution. In this vein, it is convenient that our test statistic (which is notably complex-valued) is conveniently bounded in absolute value by 1, thereby allowing for strong concentration of measure effects. 

A second application of the Fourier transform would nominally entail another integral (in the variable $\bd{\xi}$) against the complex harmonic $e^{i \langle \bd{\xi} , \bd{t} \rangle}$. However, numerically integrating statistical estimates, such as those obtained from the fist round of stochastic Fourier transform above, can be rather challenging, with attendant numerical instability effects. We tackle this issue by replacing the Lebesgue measure in this second integral by a Gaussian density with a suitable variance $\eta$; the intuitive idea being that the true integral can be seen as a limit when $\eta \to \infty$. This introduces the additional parameter $\eta$ into our estimation procedure, but this is not too difficult to eliminate, and is the focus of the subsequent step of the process, discussed in the next section. 

In this work, we combine the two steps above in one stroke, by considering the estimator in \eqref{eq:phiest}, which in turn is motivated by the expectation-level quantity $\varphi$ in \eqref{eq:phi} and \eqref{eq:phialt}. In fact, an intuitively clarifying interpretation to \eqref{eq:phi} would be to first take the expectation with respect to the random variable $\bd{X}$ (corresponding to the first round of Fourier transform discussed above), followed by expectation with respect to the variable $\bd{Y}$ (corresponding to the second round of Fourier transformation in our earlier discussion).

Let $\varphi:\bbR^d\to\bbR$ be defined as
\begin{equation} \label{eq:phi}
\varphi(\bd{t})\coloneqq\Exp[\Big]{\expo[\big]{i\langle\bd{Y},\bd{X}+\bd{t}\rangle}}
\end{equation}
where 
$\bd{X}\sim\mathcal{N}(\bd{0},\bd{\Sigma})$, 
$\bd{Y}\sim\mathcal{N}(\bd{0},\eta\bd{I})$,
and  $\langle\cdot,\cdot\rangle:\bbR^d\times\bbR^d\to\bbR$ is the usual inner product
\[
\langle\bd{y},\bd{x}\rangle \coloneqq
\sum_{j=1}^d y_j x_j\;.
\]
An alternative expression of $\varphi(\bd{t})$
is the following (see Lemma~\ref{lem:4.1})
\begin{equation} \label{eq:phialt}
\varphi(\bd{t})
= 
\det\left(\frac{1}{\eta}\Leta\right)^{1/2}
\cdot
\expo*{-\frac{1}{2} \langle \bd{t},\Leta\bd{t} \rangle}\;.
\end{equation}
Estimating $\varphi(\bd{t})$ 
for well-chosen values of $\bd{t}\in\bbR^d$, and suitably aggregating 
these estimates, we will obtain an estimate of $\Leta$. 

The sample version of the definition \eqref{eq:phi} of $\varphi(\bd{t})$ 
naturally suggests the following unbiased 
estimator:
\begin{equation} \label{eq:phiest}
\varphi_n(\bd{t})
\coloneqq
\frac{1}{n}
\sum_{k=1}^n 
\expo[\big]{i\langle\bd{Y}_k,\bd{X}_k+\bd{t}\rangle}\;,
\end{equation}
where $\bd{Y}_1,\ldots,\bd{Y}_n$ 
are i.i.d.\ random vectors in 
$\bbR^d$ which are independent 
of the $\bd{X}_j$s with the common
distribution $\mathcal{N}(\bd{0},\eta\bd{I})$.
Let $\bd{e}_1,\ldots,\bd{e}_d$ 
denote the standard basis of 
$\bbR^d$. 

We propose the following 
estimator of $\Leta = [l_{ij}^{(\eta)}]_{1\le i,j \le d}$:
\begin{equation}\label{eq:main}
\widehat{\Leta}
\coloneqq 
\Bigl[\hat{\leta_{ij}} \Bigr]_{1\le i,j \le d}\;, 
\end{equation}
where

{\begingroup
\addtolength{\jot}{1em}
\begin{align}
\hat{l_{ij}^{(\eta)}} \coloneqq & - 2\log\big|\varphi_n \left(\frac{\bd{e}_i+\bd{e}_j}{\sqrt{2}}\right)\big|
+ \log\big|\varphi_n(\bd{e}_i)\big| + \log\big|\varphi_n(\bd{e}_j)\big| \mbox{ for }i\neq j\;, \mbox{and}\\
\hat{l_{ii}^{(\eta)}} \coloneqq & 
-2\log\big|\varphi_n(\bd{e}_i)\big| + 2\log\big|\varphi_n(\bd{0})\big| \;.
\end{align}
\endgroup}

\subsection{Estimation of $\bd{\Sigma}^{-1}$ from $\Leta$}
The quantity $\Leta$ estimated in the previous section involves the parameter $\eta$ that is an artefact of our procedure. In this section our goal is to put forward a principled way of eliminating the parameter $\eta$ and obtain an estimate of $\bd{\Sigma}^{-1}$, which is our  object of interest. 

A vanilla approach to obtaining $\bd{\Sigma}^{-1}$ from $\Leta$, in light of its defining equation \eqref{eq:Leta-def}, is to observe that $\bd{\Sigma}^{-1}$ and $\Leta$ commute, and therefore spectral inversion is a possibility. However, direct inversion at the level of eigenvalues can lead to numerical instabilities, and instead, we propose the following approach, based on the so-called \textit{Woodbury's identity}. 

We first write down an identity relating $\bd{\Sigma}^{-1}$ and $\Leta$. To this end, we recall the Woodbury matrix identity \cite{woodbury1950inverting}:
\begin{equation} \label{eq:Woodbury}
\bigl(\bd{A} + \bd{U}\bd{C}\bd{V}\bigr)^{-1} 
= \bd{A}^{-1} - 
\bd{A}^{-1} \bd{U} \bigl(\bd{C}^{-1} + 
\bd{V} \bd{A}^{-1} \bd{U} \bigr)^{-1} \bd{V} \bd{A}^{-1}\;.
\end{equation}
With $\bd{A} = \eta^{-1} \bd{I}$,
$\bd{C} = \bd{\Sigma}$, 
$\bd{U} = \bd{I}$ 
and $\bd{V} = \bd{I}$,
we get
\[
\Leta 
= \eta \bd{I}  - 
\eta \bd{I} \bigl(\bd{\Sigma}^{-1} 
+ \eta \bd{I})^{-1} \eta \bd{I} 
= \eta \bd{I} - \eta^2 (\bd{\Sigma}^{-1} 
+ \eta \bd{I})^{-1}\;.
\]
Thus
\begin{equation}\label{eq:Sigma-inv-expression}
\bd{\Sigma}^{-1} = \eta^2 \bigl(\eta \bd{I} - \Leta\bigr)^{-1} - \eta \bd{I}\;.
\end{equation}
Thus from an estimator $\widehat{\Leta}$ of $\Leta$, one can construct a plug-in estimator of $\bd{\Sigma}^{-1}$:
\begin{equation}\label{eq:Sigma-inv-plugin-estimator}
    \widehat{\bd{\Sigma}^{-1}} = 
    \eta^2 \bigl(\eta \bd{I} - 
    \widehat{\Leta}\bigr)^{-1} - \eta \bd{I}\;.
\end{equation}

\subsection{A vanilla spectral estimator}
In \cite{Belomestny2019SparseDeconvolution}, the authors introduced the following estimator for the covariance matrix. 
For $\psi_n(\bd{u})$ the empirical version of the characteristic function of the Gaussian field, given by 
\[\psi_n(\bd{u}) := \frac{1}{n} \sum_{j=1}^n\exp(i \langle \bd{u} , \bd{X}_j \rangle),\] and a suitably chosen large parameter $U>0$, we define
{\begingroup
\addtolength{\jot}{1em}
\mbox{and}\\
\begin{align}
 (\hat{\bd{\Sigma}}_{\bmt})_{ii} \coloneqq & 
-\frac{2}{U^2} \Re\left(\log \psi_n(U\bd{e}_i)\right) \;, \; \mbox{and}\\
(\hat{\bd{\Sigma}}_{\bmt})_{ij}  \coloneqq & -\frac{2}{U^2} \Re\left(\log \psi_n(U \cdot \frac{\bd{e}_i+\bd{e}_j}{\sqrt{2}})\right) - \frac{1}{2} \left( (\hat{\bd{\Sigma}}_{\bmt})_{ii} + (\hat{\bd{\Sigma}}_{\bmt})_{jj}  \right) \mbox{ for }i\neq j\;.
\end{align}
\endgroup}

A naive approach to estimating $\bd{\Sigma}^{-1}$ would be to invert $\hat{\bd{\Sigma}}_{\bmt}$ directly. Conceptually, a principal point of divergence in which this estimator differs from our approach is that, instead of directly computing the matrix inverse $\bd{\Sigma}^{-1}$, we apply another round of Gaussian Fourier analysis and access $\bd{\Sigma}^{-1}$ indirectly via that route. 

Since both approaches involve Gaussian Fourier analytic or spectral ideas and complex valued statistics, it would be of interest to compare the two techniques. In particular, such comparison will clarify whether accessing the Laplacian indirectly via Gaussian Fourier analysis brings in any statistical benefits.

In summary, our analysis appears to corroborate the fact that the application of Gaussian Fourier analysis to access the Laplacian indirectly, as in our approach, brings in significant statistical benefits. As such, this makes the case for wider investigation of similar ideas in the study of Gaussian random fields in general and Gaussian Graphical Models in particular. 

\section{Theoretical guarantees}
\label{sec:theoretical}

In this section, we lay out theoretical guarantees that demonstrate the effectiveness of our method.

\subsection{Estimation rates via concentration bounds}

We first state a result showing how the error in 
estimating $\Leta$ influences the 
estimation error of $\widehat{\bd{\Sigma}^{-1}}$.

\begin{theorem}\label{thm:L-to-SigmaInv}
Suppose that 
\[
\frac{\lambda_1 + \mu + \eta}{\eta^2} \|\widehat{\Leta} - \Leta\|_2 < 1\;.
\]

Then we have
\begin{equation}\label{eq:error-bd}
\frac{1}{d}
\|\widehat{\bd{\Sigma}^{-1}} - \bd{\Sigma}^{-1}\|_F 
\leq 
\frac{(\lambda_1 + \mu + \eta)^2}{\eta^4 \Bigl(1 - \frac{\lambda_1 + \mu + \eta}{\eta^2} 
\|\widehat{\Leta} - \Leta\|_2\Bigr)} 
\frac{1}{d}
\|\widehat{\Leta} - \Leta\|_F\;.
\end{equation}
\end{theorem}

Theorem \ref{thm:L-to-SigmaInv} naturally leads to the question of concentration bounds for $\widehat{\Leta}$, which we take up in the next section.

\subsection{Concentration of $\widehat{\Leta}$}

We begin with a concentration bound for $\varphi_n$.
\begin{proposition}
[Concentration of $\varphi_n$]
For any $x>0$ and $\bd{t}\in\mathbb{R}^d$, 
we have 
\[
\Prob*{\abs*{\varphi_n(\bd{t}) - \varphi(\bd{t})} \geq x} 
\leq 4 \exp\Bigl(-\frac{3nx^2}{24+8x}\Bigr)\;\cdot
\]
In particular, for $x\in (0,1]$ and $\bd{t}\in \mathbb R^d$, we have
\[
\Prob*{\abs*{\varphi_n(\bd{t}) - \varphi(\bd{t})} \geq x} 
\leq 4 \exp\Bigl(-\frac{3}{32}nx^2\Bigr)\;\cdot
\]
\label{prop:ce1}
\end{proposition}
Let 
\begin{equation}\label{eq:defcstar}
 c_\eta 
\coloneqq \det\left(\frac{1}{\eta}\Leta\right)^{1/2}\; , \; 
 c_\ast(\eta) 
 \coloneqq \frac{1}{2} c_\eta \exp\Bigl(-{1\over 2} \|\Leta\|_2^2 \Bigr)\; ,
\end{equation}
and
\begin{equation}
     S_n(\bd{t}) \coloneqq \log\abs*{\varphi_n(\bd{t})} - \log\abs*{\varphi(\bd{t})}.
\end{equation}
This allows us to state a concentration bound for $S_n$.

\begin{proposition}[Concentration of $S_n$]
For any $\bd{t} \in \mathbb R^d, \|\bd{t}\|\le 1$
and $x\in (0,1]$, we have
\[
\Prob*{ \abs*{S_n(\bd{t})} \geq x} \leq 
3 \Prob*{ \abs*{\varphi_n(\bd{t}) - \varphi(\bd{t})} \geq c_\ast(\eta) x }\;.
\]

\noindent
Therefore, for any $\bd{t} \in \mathbb R^d, \|\bd{t}\|\le 1$
 and $x\in (0,1]$, we have
 \[
 \Prob*{ |S_n(\bd{t})| \geq x } \leq C_1 \exp( - C_2 \cdot {nc_\ast(\eta)^2} x^2 )\;,
 \]
 for some universal positive constants $C_1,C_2$.
\label{prop:ce2}
\end{proposition}

Concentration of 
$\varphi_n(\bd{t})$ 
around
$\varphi(\bd{t})$
yields 
concentration of 
$\widehat{\Leta}$
around
$\Leta$
as shown by the following lemma.

\begin{lemma}[Concentration 
of $\varphi_n$ implies 
concentration of 
$\widehat{\Leta}$]
\label{lem2.1}
We have 
{\begingroup
\addtolength{\jot}{1em}
\begin{eqnarray*}
    \leta_{ii} - \hat{\leta_{ii}} &=& 2S_n(\bd{e}_i) - 2S_n(\bd{0}) \\
    \leta_{ij} - \hat{\leta_{ij}} &=& 2S_n \Big ({\bd{e}_i + \bd{e}_j \over \sqrt 2} \Big ) - S_n(\bd{e}_i) - S_n(\bd{e}_j) \quad \text{for }i\ne j.
\end{eqnarray*}
\endgroup}
\end{lemma}

Finally, we are ready to state

\begin{theorem}[Concentration of $\widehat{\Leta}$]
For $x\in (0,1]$ we have
\[
\Prob*{{1\over d} \| \Leta - \widehat{\Leta} \|_F \geq x}
\leq d^2 \cdot C_1 \cdot \exp\Bigl( -C_2nc_\ast(\eta)^2 x^2 \Bigr)\;,
\]
for some universal positive constants $C_1,C_2$.
In other words, 
with probability at least 
$1 - d^{-c}$,
we have
\[
\frac{1}{d}
\|\Letahat - \Leta\|_F
= O\bigg( \frac{1}{c_\ast(\eta)} \sqrt{\frac{\log d}{n}} \bigg)
\;.
\]
\label{thm3.4}
\end{theorem}
\noindent Note that
\[
c_\ast(\eta) = {1\over 2} c_\eta \exp \Big (-{1\over 2}\|\Leta\|_2^2 \Big )
\]
and 
\[
\|\Leta\|_2 = \frac{\eta}{1 + \eta \lambda_{\min}(\bd{\Sigma})}\; \cdot
\]
Thus $c_\ast(\eta) \asymp c_{\eta}$. But
\begingroup
\addtolength{\jot}{1em}
\begin{align*}
c_{\eta} ={} & 
\Biggl(\prod_{j = 1}^d \frac{\lambda_j + \mu}{\lambda_j + \mu + \eta}\Biggr)^{1/2}
={}\Biggl(\prod_{j = 1}^d 
\bigg(1 + \frac{\eta}{\lambda_j + \mu}\bigg)^{-1}\Biggr)^{1/2}
={}\exp\bigg(-\frac{1}{2} \sum_{j = 1}^d \log \bigg(1 + \frac{\eta}{\lambda_j + \mu}\bigg)\bigg)\; \cdot
\end{align*}

\subsection{An explicit guarantee on the estimation rate}

It remains to combine the main results of the last two sections to provide an explicit guarantee on estimation rates.

\begin{theorem}\label{thm:combined_bound}
Suppose that 
\[
\frac{\lambda_1 + \mu + \eta}{\eta^2} \|\widehat{\Leta} - \Leta\|_2 < 1\;.
\]

Then, with probability $\ge 1 - d^{-c}$, we have
\begin{equation}\label{eq:error-bd}
\frac{1}{d}
\|\widehat{\bd{\Sigma}^{-1}} - \bd{\Sigma}^{-1}\|_F 
\leq 
C \cdot \frac{(\lambda_1 + \mu + \eta)^2}{\eta^4 \Bigl(1 - \frac{\lambda_1 + \mu + \eta}{\eta^2} 
\|\widehat{\Leta} - \Leta\|_2\Bigr)} \cdot 
 \frac{1}{c_\ast(\eta)} \sqrt{\frac{\log d}{n}}\;
\end{equation}
for some positive constant $C$.
\end{theorem}

\section{Generative models : Erd\"os-R\'enyi random graphs} \label{sec:generative}
In this section, we investigate the behaviour of our approach in the setting of some common generative models of random graphs, and compare with the vanilla spectral estimator based on \citet{Belomestny2019SparseDeconvolution}. In particular, we will focus on the setup where the base graph $G(d,p)$ is generated from the Erd\"os-R\'enyi model with edge probability 
\[
p = \Omega\Bigl(\frac{\log d}{d}\Bigr)\;.
\]
We note in passing that the above regime of connection probability ensures that a graph generated according to this 
 model is connected with high probability.

\subsection{Estimating Erd\"os-R\'enyi random graphs via our approach}
Herein we carry out a performance analysis of our approach for the  Erd\"os-R\'enyi random graph model.
 
To state our main result of this section, we denote the average degree of the graph by 
\[
\Delta_{\mathrm{avg}} 
\coloneqq (d - 1)p\;.
\]
We may now state
\begin{theorem}\label{thm:er-model}
If we choose $\eta = \Theta(p)$ and take
\[
    n > \frac{d^6 \log d}{\Delta_{\mathrm{avg}}^2}\;,
\]
then we have with probability at least $1 - d^{-c}$ that
\[
\frac{1}{d}
\|\Sigmainvhat - \Sigmainv\|_F \leq C \sqrt{\frac{d^4}{p^4} \frac{\log d}{n}},
\]
for some absolute constants $c, C > 0$.
\end{theorem}

\begin{remark} \label{rem:er-rem1}
It may be noted that in the dense regime of the Erd\"os-R\'enyi random graph model, i.e. with the connection probability $p=\Theta(1)$ (as $d \to \infty$), Theorem \ref{thm:er-model} implies a sample complexity of order $d^4\log d$.
\end{remark}

\begin{remark}
More generally, one can prove a similar result for inhomogeneous Erd\H{o}s-R\'{e}nyi random graphs with edge probability matrix $P$. There the role of the average degree $\Delta_{\mathrm{avg}}$ would be replaced by the maximum expected degree $\Delta_{\max} := \max_{i} \sum_{j} P_{ij}$.
\end{remark}

\subsection{Comparison to the vanilla spectral estimator}
In this section, we will provide a comparison of the theoretical guarantees on our approach vis-a-vis the vanilla spectral estimator motivated by \citet{Belomestny2019SparseDeconvolution}. The application of the latter estimator requires a choice of the parameter $U$. We undertake an analysis of two such possible choices -- one following the recommendation of \citet{Belomestny2019SparseDeconvolution}, and another following an improvisation tailored to our specific setting. 

For both choices (c.f. \eqref{eq:vanilla-canonical},\eqref{eq:vanilla-improvised}) of parameter $U$, it appears from the analysis in the two subsequent sections that for dense Erd\"os-R\'enyi random graphs (i.e. $p=\Theta(1)$), our approach has a better sample complexity guarantee for ill-conditioned GFFs (i.e., the mass parameter $\mu$ being small); refer to Theorem \ref{thm:er-model} and Remark \ref{rem:er-rem1}.

For our purpose, we state here a simplified version of Theorem 1 of \cite{Belomestny2019SparseDeconvolution}. Note that in \cite{Belomestny2019SparseDeconvolution}, the observations are $\bd{Y}_i = \bd{X}_i + \bd{\varepsilon}_i, i = 1,\ldots,n$, where $\bd{\varepsilon}_i$s are i.i.d. noises independent of $\bd{X}_i$s. In our setting, the noises $\bd{\varepsilon}_i = 0$.

\begin{theorem}[Theorem 1 of \cite{Belomestny2019SparseDeconvolution}] \label{thm:Belomestny}
    Assume that $\|\bd{\Sigma}\|_2 \le R$. Let $\gamma>\sqrt 2$ and $U\ge 1$ satisfy
    \begin{equation} \label{eq:BMT}
        8\gamma \sqrt{\log(ed)\over n} < e^{-RU^2}.
    \end{equation}
    Set
    \begin{equation}\label{eq:tau}
    \tau(U):= 6\gamma{e^{RU^2}\over U^2} \Big ({\log(ed)\over n} \Big )^{1/2} \cdot
    \end{equation}
    Then for any $\tau\ge \tau(U)$,
    $$\mathbb P(\|\hat{\bd{\Sigma}}_{\bmt} - \bd{\Sigma} \|_\infty < \tau ) \ge 1 - 12e^{-\gamma^2}d^{2-\gamma^2}.$$
\end{theorem}

To undertake the analysis for the vanilla estimator, we also need the following preparatory result. 
\begin{proposition}\label{prop:inv-sigma_hat}
    Let $\hat{\bd{\Sigma}}$ denote any estimator of $\bd{\Sigma}$. Assume that $s_{\min}(\bd{\Sigma}) > \|\hat{\bd{\Sigma}} - \bd{\Sigma}\|_2,$ where $s_{\min}(\bd{M})$ denotes the smallest singular value of $\bd{M}$. Then
    \[
        \frac{1}{d}\|\hat{\bd{\Sigma}}^{-1} - \bd{\Sigma}^{-1}\|_F \le \frac{\frac{1}{d}\|\hat{\bd{\Sigma}} - \bd{\Sigma}\|_F}{ s_{\min}(\bd{\Sigma}) (s_{\min}(\bd{\Sigma}) - \|\hat{\bd{\Sigma}} - \bd{\Sigma}\|_2)} \cdot
    \]
\end{proposition}

Let us now analyse the vanilla spectral  estimator $\hat{\bd{\Sigma}}_{\bmt}^{-1}$ with two specific choices of parameter $U$: the \textit{canonical choice} (recommended by Belomestny et al. in \cite{Belomestny2019SparseDeconvolution}) and an \textit{improvised choice} (tailored to our specific setting). In \cite{Belomestny2019SparseDeconvolution}, the authors suggested the \textit{canonical choice} for $U$ should take the form
\begin{equation}\label{eq:canonical}
    U = c_0 R^{-1/2}\sqrt{ \log \Big ( {n\over \log(ed)} \Big )}
\end{equation}
for some sufficiently small positive constant $c_0$. On the other hand, in the light of Theorem \ref{thm:Belomestny} (a simplified version of Theorem 1 in \cite{Belomestny2019SparseDeconvolution} which is tailored to our setting), if we fix the parameter $\gamma > \sqrt 2$ then the condition \eqref{eq:BMT} is equivalent to
\begin{equation}
    RU^2 < {1\over 2} \log \Big ({n\over \log(ed)} \Big ) - \log (8\gamma).
\end{equation}
The optimal choice for the parameter $U$ should minimize the quantity $\tau(U)$ in \eqref{eq:tau}. Thus, provided that $n$ is large enough, the \textit{improvised choice} for $U$ should be
\begin{equation}\label{eq:improvised}
    U = R^{-1/2}.
\end{equation}

\begin{theorem}\label{thm:vanilla_estimator}
(i) (Canonical parameter choice) If $U$ is of the form \eqref{eq:canonical}, 
        then for any $\alpha \in (0,1/2)$, we have
        $${1\over d} \|\hat{\bd{\Sigma}}_{\bmt}^{-1} - \bd{\Sigma}^{-1} \|_F = O \Big ({(\lambda_1 + \mu)^2 \over \mu} \Big ( {\log(ed) \over n} \Big )^{1/2 - \alpha} \Big )$$
        with probability at least $1-Cd^{-c}$ for some constants $c,C>0$, provided that
        $$  n \succeq \bigg(\frac{d(\lambda_1 + \mu)}{\mu}\bigg)^{\frac{2}{1 - 2 \alpha}}\log(ed).$$
        \medskip

\noindent        
(ii) (Improvised parameter choice) If $U$ is of the form \eqref{eq:improvised}, then we have
        $${1\over d} \|\hat{\bd{\Sigma}}_{\bmt}^{-1} - \bd{\Sigma}^{-1} \|_F = O \Big ({(\lambda_1 + \mu)^2 \over \mu} \Big ( {\log(ed) \over n} \Big )^{1/2} \Big )$$
        with probability at least $1-Cd^{-c}$ for some constants $c,C>0$, provided that
        $$  n \succeq \bigg(\frac{d(\lambda_1 + \mu)}{\mu}\bigg)^{2}\log(ed).$$
\end{theorem}

When $\bd{L}$ is the Laplacian of an ER$(d,p)$ random graph, we have 
$\lambda_1 = dp + O(\sqrt{dp\log d})$
with high probability. 
Applying Theorem \ref{thm:vanilla_estimator} to this special case gives
\begin{corollary}
        (i) (Canonical parameter choice) If $U$ is of the form \eqref{eq:canonical}, 
        then for any $\alpha \in (0,1/2)$, we have
        $${1\over d} \|\hat{\bd{\Sigma}}_{\bmt}^{-1} - \bd{\Sigma}^{-1} \|_F 
        = O \Big ({(dp)^2 \over \mu} \Big ( {\log(ed) \over n} \Big )^{1/2 - \alpha} \Big )$$
        with probability at least $1-Cd^{-c}$ for some constants $c,C>0$, provided that
        \begin{equation} \label{eq:vanilla-canonical}
          n \succeq \bigg(\frac{d^2p}{\mu}\bigg)^{\frac{2}{1 - 2 \alpha}}\log(ed).
        \end{equation}
\medskip
        
        \noindent
        (ii) (Improvised parameter choice) If $U$ is of the form \eqref{eq:improvised}, then we have
        $${1\over d} \|\hat{\bd{\Sigma}}_{\bmt}^{-1} - \bd{\Sigma}^{-1} \|_F = O \Big ({(dp)^2 \over \mu} \Big ( {\log(ed) \over n} \Big )^{1/2} \Big )$$
        with probability at least $1-Cd^{-c}$ for some constants $c,C>0$, provided that
        \begin{equation} \label{eq:vanilla-improvised} 
          n \succeq \bigg(\frac{d^2p}{\mu}\bigg)^{2}\log(ed).
        \end{equation}
\end{corollary}

\color{black}

\section{Detailed proofs of theoretical results}

In this section, we provide detailed proofs of various theoretical results in the earlier sections of the paper. 

\subsection{On the expectation of $\varphi_n(\bd{t})$}
We begin with a Lemma that deals with the expectation of $\varphi_n(\bd{t})$.
\begin{lemma}[Expectation of $\varphi_n(\bd{t})$]
\[
\varphi(\bd{t}) 
= 
\det\left(\frac{1}{\eta}\Leta\right)^{1/2}
\cdot
\expo*{-\frac{1}{2} \langle \bd{t},\Leta\bd{t} \rangle}\;.
\]
\label{lem:4.1}
\end{lemma}

\begin{proof}
We observe that
\begin{eqnarray*}
    \varphi(\bd{t}) &=& \mathbb E \Big [\exp \Big (i \langle \bd{Y}_1, \bd{X}_1 + \bd{t} \rangle \Big ) \Big ] \\
    &=& \mathbb E_{\bd{Y}_1} \Big [\exp \Big (i\langle \bd{Y}_1,\bd{t}\rangle \Big ) \mathbb E_{\bd{X}_1} \Big [\exp \Big (i\langle \bd{Y}_1, \bd{X}_1 \rangle \Big ) \Big ] \Big ] \\
    &=& \mathbb E_{\bd{Y}_1} \Big [\exp \Big (i\langle \bd{Y}_1,\bd{t}\rangle \Big ) \exp \Big ( -{1\over 2} \langle \bd{Y}_1 , \bd{\Sigma} \bd{Y}_1 \rangle \Big ) \Big ]\\
    &=& \int_{\RR^d} 
\frac{1}{\sqrt{\det (2\pi\eta\bd{I})}}
\expo*{i\langle\bd{y},\bd{t}\rangle} 
\expo*{-\frac{1}{2}\langle \bd{y}, \bd{\Sigma}\bd{y} \rangle}
\expo*{-\frac{1}{2}\langle \bd{y}, \eta^{-1} \bd{y} \rangle}
\;\deri\bd{y} \\
&=& \int_{\mathbb{R}^d} 
\frac{1}{\sqrt{\det(2\pi\eta\bd{I})}}
\expo*{i\langle\bd{y},\bd{t}\rangle}
\expo*{-\frac{1}{2}\langle\bd{y},(\bd{\Sigma}+\eta^{-1}\bd{I})\bd{y}\rangle}
\deri\bd{y}\\
&=& \frac{1}{\sqrt{\det(\bd{I}+\eta\bd{\Sigma})}} 
\expo*{-\frac{1}{2}\langle\bd{t},\Leta\bd{t}\rangle}\\ 
&=& c_\eta \expo*{-\frac{1}{2}\langle\bd{t},\Leta\bd{t}\rangle}\;,
\end{eqnarray*}
where 
\begin{equation}\label{eq:defceta}
c_\eta 
\coloneqq 
\det\left(\bd{I}+\eta\bd{\Sigma}\right)^{-1/2} 
= \det\left(\frac{1}{\eta}\Leta\right)^{1/2}\;.
\end{equation}

\end{proof}

\subsection{Proof of Lemma~\ref{lem2.1}}
We continue with the proof of Lemma~\ref{lem2.1}.
\begin{proof}[Proof of Lemma~\ref{lem2.1}]

Observe that
\begingroup
\addtolength{\jot}{1em}
\begin{align*}
\frac{1}{2}\langle\bd{t},
\Leta\bd{t}\rangle 
{} = & - \log\varphi(\bd{t}) + \log\varphi(\bd{0})\\
{} = & - \log\varphi_n(\bd{t}) + \Bigl(\log\varphi_n(\bd{t})-\log\varphi(\bd{t})\Bigr)
+ \log\varphi(\bd{0})\;.
\end{align*}
\endgroup
Taking real parts of both sides yields
\begingroup
\addtolength{\jot}{1em}
\begin{align*}
\frac{1}{2} \langle \bd{t}, 
\Leta \bd{t} \rangle 
={} & - \log\abs*{\varphi_n(\bd{t})} + 
        \fbr[\Big]{\log\abs*{\varphi_n(\bd{t})}-\log\abs*{\varphi(\bd{t})}} + 
        \log\abs*{\varphi(\bd{0})}\\
 ={} & - \log\abs*{\varphi_n(\bd{t})} + S(\bd{t}) + \log\abs*{\varphi(\bd{0})}.
\end{align*}
\endgroup

Let $\leta_{ij}$ denote the entry at $i$'th row and $j$'th column of $\Leta$.
Since $\Leta$ is symmetric,
we have 
\[
\langle \bd{e}_i,\Leta\bd{e}_i 
\rangle 
= \leta_{ii}\;,
\quad \Big \langle {\bd{e}_i +\bd{e}_j \over \sqrt 2}, \Leta \Big ({\bd{e}_i+\bd{e}_j \over \sqrt 2} \Big ) \Big \rangle 
= {1\over 2}\leta_{ii} +  \leta_{ij} + {1\over 2} \leta_{jj} \quad \text{for $i\neq j$}\;.
\]

Recall that
\begin{eqnarray*}
    \hat{\leta_{ii}} &=& -2 \log |\varphi_n(\bd{e}_i)| + 2 \log |\varphi_n(\bd{0})| \\
    \hat{\leta_{ij}} &=& -2 \log \Big | \varphi_n \Big ({\bd{e}_i + \bd{e}_j \over \sqrt 2}\Big ) \Big | + \log |\varphi_n (\bd{e}_i)| + \log |\varphi_n(\bd{e}_j)| \quad\text{for } i\ne j.
\end{eqnarray*}
Thus,
\begin{eqnarray*}
\leta_{ii} - \hat{\leta_{ii}} &=& \langle \bd{e}_i , \Leta \bd{e}_i \rangle - (-2\log |\varphi_n(\bd{e}_i)| + 2 \log |\varphi_n(\bd{0})|)      \\
&=& 2S_n(\bd{e}_i) - 2S_n(\bd{0})
\end{eqnarray*}
and for $i\ne j$
\begin{eqnarray*}
\leta_{ij} - \hat{\leta_{ij}} &=& 
\Big \langle {\bd{e}_i +\bd{e}_j \over \sqrt 2}, \Leta \Big ({\bd{e}_i+\bd{e}_j \over \sqrt 2} \Big ) \Big \rangle
- {1\over 2}\Big (\leta_{ii} + \leta_{jj}\Big ) - \hat{\leta_{ij}} \\
&=& 2S_n \Big ({\bd{e}_i + \bd{e}_j \over \sqrt 2} \Big ) - S_n(\bd{e}_i) - S_n(\bd{e}_j).
\end{eqnarray*}

\end{proof}

\subsection{Proofs of Propositions~\ref{prop:ce1} and \ref{prop:ce2}}
Our next item is the proof of Proposition~\ref{prop:ce1}.
\begin{proof}[Proof of Proposition~\ref{prop:ce1}]
Note that
\[
\abs*{\varphi_n(\bd{t})-\varphi(\bd{t})} \leq 
\abs*{\Re\bigl(\varphi_n(\bd{t}) - \varphi(\bd{t})\bigr)} 
+ \abs*{\Im\bigl(\varphi_n(\bd{t})-\varphi(\bd{t})\bigr)}\;.
\]
For $k \in \{1,\ldots,n\}$, we define 
\[
\xi_k \coloneqq 
\Re\Bigl( e^{i\langle \bd{Y}_k, \bd{X}_k + \bd{t}\rangle} - \varphi(\bd{t}) \Bigr)\;.
\]
Then we have i.i.d.\ real random variables $\xi_1,\ldots,\xi_n$ such that
\[
\Re\Bigl( \varphi_n(\bd{t})-\varphi(\bd{t}) \Bigr) 
= \frac{1}{n} \sum_{k=1}^n \xi_k\;.
\]

Observe that
\[
\Exp*{\xi_k} = 0\;, \quad 
\abs*{\xi_k} \leq 1 + c_\eta \leq 2\;, \quad
\Var*{\xi_k} \leq 1 - \abs*{\varphi(\bd{t})}^2 \leq 1\;.
\]
By the Bernstein's inequality, we have for any $x>0$
\[
\Prob*{\abs*{\Re\bigl(\varphi_n(\bd{t})-\varphi(\bd{t}))}\geq x} =
\Prob*{\abs*{\frac{1}{n} \sum_{k=1}^n \xi_k}  \geq x} \leq 
2 \exp\Bigl(-\frac{\frac{1}{2} n x^2}{1+\frac{2}{3}x}\Bigr)\; \cdot
\]
Using similar bound for the imaginary part, we have
\begingroup
\addtolength{\jot}{1em}
\begin{align*}
{} & 
\Prob[\Big]{\abs*{\varphi_n(\bd{t})-\varphi(\bd{t})}\geq x } \\ 
\leq{} &
\Prob[\Big]{\abs*{\Re\bigl(\varphi_n(\bd{t})-\varphi(\bd{t})\bigr)} 
+ \abs*{\Im\bigl(\varphi_n(\bd{t})-\varphi(\bd{t})\bigr)} \geq x}\\
\leq{} & 
\Prob[\Big]{\abs*{\Re\bigl(\varphi_n(\bd{t})-\varphi(\bd{t})\bigr)}
\geq 
\frac{x}{2}} 
+ \Prob[\Big]{\abs*{\Im\bigl(\varphi_n(\bd{t})-\varphi(\bd{t})\bigr)} \geq \frac{x}{2}}\\
\leq{} & 
4 \exp\Bigl(-\frac{\frac{1}{8}nx^2}{1+\frac{1}{3}x}\Bigr)
= 
4 \exp\Bigl(-\frac{3nx^2}{24+8x}\Bigr)\; \cdot
\end{align*}
\endgroup
This completes the proof of Proposition~\ref{prop:ce1}.
\end{proof}

We continue on to the proof of Proposition~\ref{prop:ce2}.
\begin{proof}[Proof of Proposition~\ref{prop:ce2}]
We have
\[
S_n(\bd{t}) 
= \log\Bigl|\frac{\varphi_n(\bd{t})}{\varphi(\bd{t})}\Bigr| 
\leq 
   \log\Bigl(\Bigl|\frac{\varphi_n(\bd{t})-\varphi(\bd{t})}{\varphi(\bd{t})}\Bigr| + 1\Bigr) 
\leq \Bigl|\frac{\varphi_n(\bd{t})-\varphi(\bd{t})}{\varphi(\bd{t})}\Bigr|\;,
\]
which implies that 
\[
\Prob[\Big]{ S_n(\bd{t})\geq x } \leq 
\Prob[\Big]{ \abs*{\varphi_n(\bd{t})-\varphi(\bd{t}) } \geq x \cdot \varphi(\bd{t})} 
\quad \text{for any }x>0\;.
\]  
On the event 
\[
\Set*{ \abs*{\varphi_n(\bd{t}) - \varphi(\bd{t})} \leq \frac{1}{2} \abs*{\varphi(\bd{t})} }
\]
we have
\begingroup 
\addtolength{\jot}{1em}
\begin{align*}
-S_n(\bd{t}) = {} & \log\Bigl|\frac{\varphi(\bd{t})}{\varphi_n(\bd{t})}\Bigr| 
\leq {}  \log\Bigl(\Bigl|\frac{\varphi_n(\bd{t})-\varphi(\bd{t})}{\varphi_n(\bd{t})}\Bigr|
+1\Bigr)  \\
\leq {} &  \log\Bigl(2\Bigl|\frac{\varphi_n(\bd{t})-\varphi(\bd{t})}{\varphi(\bd{t})}\Bigr| 
+ 1\Bigr)
\leq {}  2 \Bigl|\frac{\varphi_n(\bd{t})-\varphi(\bd{t})}{\varphi(\bd{t})}\Bigr|\; \cdot
\end{align*}
\endgroup
Hence, for any $x>0$,
\[
\Prob[\Big]{-S_n(\bd{t})\geq x} \leq
\Prob[\Big]{\abs*{\varphi_n(\bd{t}) - \varphi(\bd{t})} \geq 
\frac{x}{2}\abs*{\varphi(\bd{t})}} + 
\Prob[\Big]{\abs*{\varphi_n(\bd{t}) - \varphi(\bd{t})} > 
\frac{1}{2}\abs*{\varphi(\bd{t})}}\;.
\]
In particular, if $x\in (0,1]$ we deduce that
\[
\Prob*{ \abs{S_n(\bd{t})} \geq x } \leq 3 \cdot 
\Prob[\Big]{\abs*{\varphi_n(\bd{t}) - \varphi(\bd{t})} \geq \frac{x}{2} 
\abs*{\varphi(\bd{t})}}\;.
\]

 On the other hand, for any $\bd{t} \in \mathbb R^d$ with $\|\bd{t}\| \le 1$, one has
$$ \abs*{\langle\bd{t},\Leta \bd{t}\rangle}  \le \|\Leta\|_2^2.$$
This implies
$$|\varphi(\bd{t})| = c_{\eta} \exp\Bigl(-\frac{1}{2}\langle\bd{t},\Leta\bd{t}\rangle\Bigr) \le c_{\eta} \exp \Bigl(-{1\over 2} \|\Leta\|_2^2 \Bigr ) = 2 c_*(\eta).$$
Thus, for $x\in (0,1]$ and $\|\bd{t}\|\le 1$
\[
\Prob*{ |S_n(\bd{t})| \geq x} 
\leq 3 \Prob*{ |\varphi_n(\bd{t}) - \varphi(\bd{t})| \geq c_\ast(\eta)x }\;.
\]
\end{proof}

\subsection{Proofs of Theorems~\ref{thm:L-to-SigmaInv} and \ref{thm3.4}}
We first tackle the proof of Theorem~\ref{thm3.4}
\begin{proof}[Proof of Theorem~\ref{thm3.4}]
For $x\in (0,1]$, we have
\begingroup
\addtolength{\jot}{1em}
\begin{align*}
\Prob[\Big]{ \abs*{ \leta_{ii} - 
\hat{\leta_{ii}} } \geq x } 
\leq{} & \Prob[\Big]{ 2\abs*{ S_n(\bd{e}_i) } + 2| S_n(\bd{0}) | \geq x } \\
\leq{} & \Prob[\Big]{ \abs*{ S_n(\bd{e}_i) } \geq {x\over 4} } + 
\Prob[\Big]{ \abs*{ S_n(\bd{0}) } \geq {x\over 4} } \\
\leq{} & C_1\cdot\exp\Bigl(-C_2 nc_\ast(\eta)^2 x^2\Bigr)\;,
\end{align*}
\endgroup
for some universal constants $C_1,C_2>0$. 
\medskip

For $i\ne j$ and $x\in (0,1]$, we have
\begingroup
\addtolength{\jot}{1em}
\begin{align*}
\Prob[\Big]{\abs[\big]{l_{ij}^{(\eta)}-\hat{ l_{ij}^{(\eta)}}}\geq x} \leq{} & 
\Prob[\Big]{2\abs[\big]{S_n\Big ({\bd{e}_i+\bd{e}_j\over \sqrt 2} \Big)}
+\abs[\big]{S_n(\bd{e}_i)}+\abs[\big]{S(\bd{e}_j)} \geq x }\\
\leq{} & \Prob[\Big]{\abs[\big]{S_n \Big ({\bd{e}_i+\bd{e}_j \over \sqrt 2} \Big )}\geq\frac{x}{4}} 
+ \Prob[\Big]{\abs[\big]{S_n(\bd{e}_i)} \geq \frac{x}{4}} 
+ \Prob[\Big]{\abs[\big]{S_n(\bd{e}_j)} \geq \frac{x}{4}} \\
\leq{} & C_1 \cdot \exp\Bigl( -C_2nc_\ast(\eta)^2 x^2 \Bigr)\;
\end{align*}
\endgroup
for some universal constants $C_1,C_2>0$.
\medskip

Therefore, for $x\in (0,1]$
\[
\Prob*{\max_{i,j}
\abs*{\leta_{ij}-\hat{\leta_{ij}}} \geq x }
\leq d^2 \cdot C_1 \cdot \exp\Bigl( -C_2nc_\ast(\eta)^2 x^2 \Bigr)\;,
\]
for some universal constants $C_1,C_2>0$.
Note that,
\[
\|\Leta - \Letahat \|_F^2 = 
\sum_{i,j} \abs*{
\leta_{ij} - 
\hat{\leta_{i,j}} }^2 
\leq 
d^2 \cdot \max_{i,j} 
\abs*{\leta_{ij} - \hat{\leta_{ij}} }^2\;.
\]

\end{proof}

\color{black}

We are now ready to address the proof of Theorem~\ref{thm:L-to-SigmaInv}.
\begin{proof}[Proof of Theorem~\ref{thm:L-to-SigmaInv}]
From \eqref{eq:Sigma-inv-expression} and \eqref{eq:Sigma-inv-plugin-estimator}, we have
\[
\widehat{\bd{\Sigma}^{-1}} - 
\bd{\Sigma}^{-1} = 
\eta^2
\Bigl(\bigl(\eta\bd{I} - \widehat{\Leta}\bigr)^{-1} - 
\bigl(\eta\bd{I} - \Leta\bigr)^{-1}\Bigr)
\;.
\]
Writing $\bd{X} = \Leta$ and $\bd{X}^\prime = \Letahat$, we have
\begingroup
\addtolength{\jot}{1em}
\begin{align*}
{} & \Bigl(\eta\bd{I} - \bd{X}^\prime\Bigr)^{-1} 
={}  \Bigl(\eta\bd{I} - \bd{X} + \bd{X} - \bd{X}^\prime\Bigr)^{-1} \\
={} & \Bigl(\eta\bd{I} - \bd{X}\Bigr)^{-1} 
- \Bigl(\eta\bd{I} - \bd{X}\Bigr)^{-1} 
\Bigl(\bd{X} - \bd{X}^\prime\Bigr)
\Bigl(\bd{I} + \bigl(\eta\bd{I} - \bd{X}\bigr)^{-1} \bigl(\bd{X} - \bd{X}^\prime\bigr)\Bigr)^{-1}
\Bigl(\eta\bd{I} - \bd{X}\Bigr)^{-1}\;,
\end{align*}
\endgroup
where for the second equality 
above we have used the Woodbury 
identity with $\bd{A} = \eta\bd{I} - \bd{X}$,
$\bd{C} = \bd{I}$, $\bd{U} = \bd{X} - \bd{X}^\prime$ and $\bd{V} = \bd{I}$. 
Now, using the inequalities 
$\|\bd{A}\bd{B}\|_F \le \|\bd{A}\|_F \|\bd{B}\|_2$ and $\|\bd{A} \bd{B}\|_2
\le \|\bd{A}\|_2 \|\bd{B}\|_2$, we obtain
\begingroup
\addtolength{\jot}{1em}
\begin{align*}
{}&\|(\eta\bd{I} - \bd{X}^\prime)^{-1}
- (\eta\bd{I} - \bd{X})^{-1}\|_F\\
={}&\|(\eta\bd{I} - \bd{X})^{-1} 
(\bd{X} - \bd{X}^\prime) (\bd{I} + 
(\eta\bd{I} - \bd{X})^{-1} (\bd{X} - \bd{X}^\prime))^{-1} 
(\eta \bd{I} - \bd{X})^{-1}\|_F \\
\leq{}&\|(\eta \bd{I} - \bd{X})^{-1}
(\bd{X} - \bd{X}^\prime)\|_F 
\|(\bd{I} + (\eta\bd{I} - \bd{X})^{-1}
(\bd{X} - \bd{X}^\prime))^{-1}
(\eta\bd{I} - \bd{X})^{-1}\|_2 \\
\le{}&\|(\eta\bd{I} - \bd{X})^{-1}\|_2
\|(\bd{X} - \bd{X}^\prime)\|_F 
\|(\bd{I} + (\eta\bd{I} - \bd{X})^{-1}
(\bd{X} - \bd{X}^\prime))^{-1}\|_2
\|(\eta\bd{I} - \bd{X})^{-1}\|_2 \\
\leq{}&\|\bd{X} - \bd{X}^\prime\|_F
\|(\eta \bd{I} - \bd{X})^{-1}\|_2^2
\|(\bd{I} + (\eta\bd{I} - \bd{X})^{-1} 
(\bd{X} - \bd{X}^\prime))^{-1}\|_2\;.
\end{align*}
\endgroup

Now observe that
$
\eta\bd{I} - \bd{X} 
= \eta\bd{I} - \Leta
= \eta^2 
(\Sigmainv + \eta \bd{I})^{-1}
$
, which implies
$
\bigl(\eta\bd{I} - \bd{X}\bigr)^{-1} 
= \eta^{-2} 
\bigl(\Sigmainv + \eta\bd{I}\bigr)\;.
$
Hence
\[
\|\bigl(\eta\bd{I} - \bd{X}\bigr)^{-1}\|_2  = \eta^{-2} 
\bigl(\lambda_1 + \mu + \eta\bigr)\;.
\]
\medskip

Let
$
\bd{E} := (\eta\bd{I} - \bd{X})^{-1} (\bd{X} - \bd{X}^\prime)\;.
$
Then
\[
\|\bd{E}\|_2 
\le \eta^{-2} 
\bigl(\lambda_1 + \mu + \eta\bigr) 
\|\bd{X} - \bd{X}^\prime\|_2\;.
\]
As long as $\|\bd{E}\|_2 < 1$,
we have
\[
\|(\bd{I} + \bd{E})^{-1}\|_2
\le \frac{1}{1 - \|\bd{E}\|_2}\; \cdot
\]

{Putting everything together, we get
\begingroup
\addtolength{\jot}{1em}
\begin{eqnarray*}
{}\frac{1}{d}
\|\Sigmainvhat -
\Sigmainv\|_F
{}&\le&
\frac{(\lambda_1 + \mu + \eta)^2}{\eta^4 (1 - \|\bd{E}\|_2)} \cdot
\frac{1}{d} \cdot
\|\Letahat- \Leta\|_F\\
{}&\le&
\frac{(\lambda_1 + \mu + \eta)^2}{\eta^4 \big(1 - \frac{\lambda_1 + \mu + \eta}{\eta^2} \|\Letahat - \Leta\|_2\big)} \cdot
\frac{1}{d} \cdot
\|\Letahat - \Leta\|_F\;.
\end{eqnarray*}
\endgroup}
This completes the proof.
\end{proof}

\subsection{Proof of Proposition \ref{prop:inv-sigma_hat}}
We now establish Proposition \ref{prop:inv-sigma_hat}, which is of importance for the main results in Section \ref{sec:generative}.
\begin{proof}[Proof of Proposition \ref{prop:inv-sigma_hat}]
    The identity
    \[
        \hat{\bd{\Sigma}}^{-1} - \bd{\Sigma}^{-1} = \hat{\bd{\Sigma}}^{-1}(\bd{\Sigma} - \hat{\bd{\Sigma}}) \bd{\Sigma}^{-1}
    \]
    gives
    \begin{align} \nonumber
        \|\hat{\bd{\Sigma}}^{-1} - \bd{\Sigma}^{-1}\|_F &= \| \hat{\bd{\Sigma}}^{-1}(\bd{\Sigma} - \hat{\bd{\Sigma}}) \bd{\Sigma}^{-1}\|_F \\ \nonumber
        &\le \|\hat{\bd{\Sigma}}^{-1}\|_2 \|(\bd{\Sigma} - \hat{\bd{\Sigma}}) \bd{\Sigma}^{-1}\|_F \\ \nonumber
        &\le \|\hat{\bd{\Sigma}}^{-1}\|_2 \|\bd{\Sigma} - \hat{\bd{\Sigma}}\|_F \|\bd{\Sigma}^{-1}\|_2 \\ \label{eq:mat-inv-lipschitz-const}
        &\le \frac{\|\hat{\bd{\Sigma}} - \bd{\Sigma}\|_F}{s_{\min}(\hat{\bd{\Sigma}}) s_{\min}(\bd{\Sigma})} \cdot
    \end{align}
    Now, by Weyl's inequality,
    \[
        |s_{\min}(\hat{\bd{\Sigma}}) - s_{\min}(\bd{\Sigma})| \le \|\hat{\bd{\Sigma}} - \bd{\Sigma}\|_2,
    \]
    which yields the lower bound
    \[
        s_{\min}(\hat{\bd{\Sigma}}) \ge s_{\min}(\bd{\Sigma}) - \|\hat{\bd{\Sigma}} - \bd{\Sigma}\|_2 > 0.
    \]
    Plugging this into \eqref{eq:mat-inv-lipschitz-const} we get the desired upper bound.
\end{proof}

\subsection{Proofs of Theorems~\ref{thm:er-model} and \ref{thm:vanilla_estimator}}
We complete this section with the proofs of the Theorems in Section \ref{sec:generative}.
\begin{proof}[Proof of Theorem~\ref{thm:er-model}]
For semi-dense Erd\"os-R\'enyi graphs, i.e.,
\[
p = \Omega\Bigl(\frac{\log d}{d}\Bigr)\; ,
\]
all the non-zero eigenvalues of $\bd{L} = \bd{D} - \bd{A}$ are $\Theta(dp)$, with probability at least $1 - d^{-c}$. This follows from Weyl's inequality that
\[
|\lambda_i - (d - 1) p| 
\le 
\|(\bd{D} - \bd{A}) - 
((d - 1)p\bd{I} - p(\bd{J} - \bd{I}))\|_2 
\]
for $i = 1, \ldots, d - 1$, 
and the estimate
\begingroup
\addtolength{\jot}{1em}
\begin{align*}   
{} & \|\bd{D} - \bd{A} - ((d - 1)p\bd{I} - p(\bd{J} - \bd{I}))\|_2\\
\le{} & \|\bd{D} - (d - 1)p\bd{I}\|_2
+ \|\bd{A} - p(\bd{J} - \bd{I})\|_2\\ 
={}& O\Bigl(\sqrt{dp\log d}\Bigr)\;,
\end{align*}
\endgroup
which holds with probability at least $1 - d^{-c}$.

Indeed,
\[
\|\bd{D} - (d - 1)p\bd{I}\|_2 
= 
\max_i |D_{ii} - (d - 1)p|\;.
\]
Now $D_{ii}$ follows Binomial
$((d - 1),p)$,
and hence $D_{ii} - (d - 1)p$
is zero mean Sub-Gaussian with parameter $\sigma^2 = O(dp)$,
whence it follows (c.f. \citet{vershynin-book} Chap. 2) that with probability at least $1 - d^{-c}$ we have 
\[
\max_i |D_{ii} - (d - 1)p| = O\bigl(\sqrt{dp\log d}\bigr)\;.
\]
On the other hand,
\[
\|\bd{A} - p(\bd{J} - \bd{I})\|_2
\le 
O\Bigl(\sqrt{dp}\Bigr)\;,
\]
with probability at least
$1 - d^{-c}$ 
(see, e.g., Theorem 5.2 in \cite{lei2015consistency}).

Thus if $\eta = O(p)$,  
we have
\[
    c_{\eta} \ge \exp\bigg(-\frac{1}{2} \sum_{j = 1}^d \frac{\eta}{\lambda_j + \mu}\bigg) \asymp \exp\bigg(-C'\frac{\eta}{p}\bigg) = \Theta(1).
\]

Therefore, with probability at least $1 - d^{-c}$, we have
\[
    \frac{1}{d} \|\hat{\Leta} - \Leta\|_F = O\bigg(\sqrt{\frac{\log d}{n}} \bigg).
\]

{
Using \eqref{eq:error-bd}, we obtain
\begingroup
\addtolength{\jot}{1em}
\begin{align*}
{}&\frac{1}{d}
\|\Sigmainvhat - \Sigmainv\|_F\\ 
\le{}&
\frac{(dp)^2}{p^4} 
\frac{1}{1-
\frac{d\|\Letahat - \Leta\|_2}{p}} \frac{1}{d} 
\|\Letahat - \Leta\|_F\\
={}&
O\bigg(\sqrt{\frac{d^4}{p^4}\frac{\log d}{n}}\bigg)\;,
\end{align*}
\endgroup
}
provided
\[
\frac{d\|\Letahat - \Leta\|_2}{p} \le 
\frac{d^2}{p} 
\frac{1}{d} 
\|\Letahat - \Leta\|_F
\le 
\frac{d^2}{p} 
\sqrt{\frac{\log d}{n}} < 1\;.
\]
Thus we need a sample complexity of
\[
    n > \frac{d^6 \log d}{\Delta_{\mathrm{avg}}^2}
\]
to attain an estimation error of $O(\sqrt{\frac{d^4}{p^4} \frac{\log d}{n}}).$ 
\end{proof}

We finally complete the proof of Theorem~\ref{thm:vanilla_estimator}.
\begin{proof}[Proof of Theorem~\ref{thm:vanilla_estimator}]
First, we note that in our set-up,
\[
   \|\bd{\Sigma}\|_2 = {\mu}^{-1} \quad,\quad s_{\min}(\bd{\Sigma}) = (\lambda_1 + \mu)^{-1}.
\]
\noindent
   (i) From Theorem~1 of \cite{Belomestny2019SparseDeconvolution}, we get by choosing $R = \|\bd{\Sigma}\|_{2} = \mu^{-1}$ and $U = c_0 \sqrt{\frac{1}{R}\log(n / \log(ed))}$ (where $c_0 > 0$ is a sufficiently small constant) that for any $\alpha \in (0, 1/2)$,
\[
    \|\hat{\bd{\Sigma}}_{\bmt} - \bd{\Sigma}\|_{\infty} = O\bigg(\mu^{-1}\bigg(\frac{\log(e d)}{n}\bigg)^{1/2 - \alpha}\bigg),
\]
with probability at least $1 - C d^{-c}$ for some $c, C > 0$. As $\|\cdot\|_2 \le \|\cdot\|_F \le d\|\cdot\|_{\infty}$, we get from Proposition~\ref{prop:inv-sigma_hat} that
\[
        \frac{1}{d}\|\hat{\bd{\Sigma}}_{\bmt}^{-1} - \bd{\Sigma}^{-1}\|_F = O\bigg((\lambda_1 + \mu)^2 \mu^{-1} \bigg(\frac{\log(e d)}{n}\bigg)^{1/2 - \alpha}\bigg),
\]
provided
\[
    d \mu^{-1} \bigg(\frac{\log(e d)}{n}\bigg)^{1/2 - \alpha} < c_* s_{\min}(\bd{\Sigma}) = c_* (\lambda_1 + \mu)^{-1} ,
\]
for a sufficiently small $c_* > 0$,
i.e.
\[
    n \succeq \bigg(d\mu^{-1}(\lambda_1 + \mu)\bigg)^{\frac{2}{1 - 2 \alpha}}\log(ed),
\]
as desired.
\medskip

\noindent
(ii) From Theorem~1 of \cite{Belomestny2019SparseDeconvolution}, we get by choosing $R = \|\bd{\Sigma}\|_{2} = \mu^{-1}$ and $U=R^{-1/2}$ that
$$\|\hat{\bd{\Sigma}}_{\bmt}-\bd{\Sigma}\|_\infty = O\Big (\mu^{-1} \Big({\log(ed)\over n}\Big )^{1/2}\Big )$$
with probability at least $1-Cd^{-c}$ for some $c,C>0$. Using similar argument as part (i), we get the desired result. This completes the proof.
\end{proof}

\setlength\bibsep{10pt}
\bibliographystyle{plainnat}
\bibliography{references}

\end{document}